\documentclass[reqno,12pt]{amsart}
\usepackage{amssymb,amsmath}

\newtheorem{theorem}{Theorem}
\newtheorem{proposition}{Proposition}
\newtheorem{conjecture}{Conjecture}
\newtheorem{problem}{Problem}
\newtheorem{lemma}{Lemma}
\newtheorem{corollary}{Corollary}
\theoremstyle{definition}

\begin{document}

\title{On the isomorphism problem of Cayley digraphs}
\author{P\'eter P.~P\'alfy}
\address{HUN-REN Alfr\'ed R\'enyi Institute of Mathematics,
Re\'altanoda utca 13--15, 1053 Budapest, Hungary and
E\"otv\"os University, Department of Algebra and Number Theory}
\email{ppp@renyi.hu}
\thanks{This research was supported by the National Research, Development and
Innovation Office of Hungary (NKFIH), grant no.~153681}

\begin{abstract}
A finite group $G$ is called a DCI-group provided any two Cayley digraphs $Cay(G,S)$ and $Cay(G,T)$ are isomorphic iff the two connection sets $S$ and $T$ are images of each other via automorphisms of the group $G$. Building on the vast literature on CI- and DCI-groups we present an updated list of possible DCI-groups. In addition, we give an
alternative proof that does not rely on deep results from group theory. We also point out the three most important open problems related to the classification of DCI-groups.
\end{abstract}

\maketitle

\section{Results and open problems}

If $\alpha$ is an automorphism of the finite group $G$ and $S\subset G$, then the Cayley digraphs $Cay(G,S)$ and $Cay(G,S^\alpha)$ are obviously isomorphic. If the converse is also true, i.e., if any two Cayley digraphs of the group $G$, $Cay(G,S)$ and $Cay(G,T)$ can be isomorphic only when $T$ is the image of $S$ under some group automorphism, then we say that $G$ is a DCI-group. If we require this property only for undirected graphs (that is, whenever $S^{-1}=S$ and $T^{-1}=T$) then $G$ is called a CI-group. The definition goes back to L\'aszl\'o Babai \cite{Babai}, who was inspired by a research problem published by Andr\'as \'Ad\'am \cite{Adam} in 1967. Determining which groups are CI- or DCI-groups is still an open problem, despite the extensive literature devoted to studying Cayley (di)graphs from this point of view. Some recent progress, especially a  breakthrough result by Ted Dobson, Mikhail Muzychuk, and Pablo Spiga \cite{DobsonMuzySpiga} makes it possible to restrict further the class of candidates satisfying these elusive properties. In \cite{DobsonMuzySpiga} a list of putative CI-groups is given. However, the subtle differences between CI-groups and DCI-groups make it more promising to deal with DCI-groups first. For example, there exist exactly three cyclic CI-groups that are not DCI-groups, as it is established by the most prominent result in this area of research, the determination of cyclic CI- or DCI-groups due to Mikhail Muzychuk \cite{Muzy}, that gives the solution to \'Ad\'am's original problem.

\begin{theorem}\label{Muzy}
\textrm{(Muzychuk \cite{Muzy})} The cyclic group $C_n$ of order $n$ is a DCI-group iff $n=2^km$, where $k\in\{0,1,2\}$ and $m$ is an odd square-free integer. In addition, the cyclic groups of order $8$, $9$, and $18$ are CI-groups (but not DCI-groups).
\end{theorem}

If the following conjecture is true, then it will be enough to consider coprime indecomposable groups, i.e., such groups that cannot be written as a direct product of two proper subgroups of coprime orders.

\begin{conjecture}
(Kov\'acs and Muzychuk \cite{KovacsMuzy}; Dobson \cite[Conjecture 43]{Dobson}) If $G_1$ and $G_2$ are DCI-groups, such that $(|G_1|,|G_2|)=1$, then $G_1\times G_2$ is a DCI-group as well.
\end{conjecture}

This is clearly not true for CI-groups, as the examples of $C_8\times C_3$ and $C_9\times C_4$ show. There are several cases when the validity of this conjecture has been verified: $C_p^4\times C_q$ ($p\ne q$ primes --- Kov\'acs and Ryabov~\cite{KovacsRyabov}), $C_p^2\times C_m$ ($p$ a prime, $m$ square-free and not divisible by $p$ --- Kov\'acs et al.~\cite{Kovacs5}), $C_p^2\times C_q^2$ ($p\ne q$ primes --- Kov\'acs et al.~\cite{Kovacs5}), $Q_8\times C_p$ ($Q_8$ the quaternion group of order and $8$, $p$ an odd prime --- Somlai~\cite{Somlai2}). These results had several predecessors. However, since subgroups of DCI-groups are DCI-groups themselves, we do not list here those earlier results that follow from the ones mentioned above.

The main result of the present paper shows that the structure of coprime indecomposable DCI-groups is severely restricted.

\begin{theorem}\label{main}
If $G$ is a coprime indecomposable DCI-group, then one of the following occurs:
\begin{enumerate}
  \item[(a)] $G=C_p^r$ is an elementary abelian $p$-group for some prime $p$ and $r\ge1$;
  \item[(b1)] $G=D_{2m}=\langle a,b\,|\, a^m=1, b^2=1, b^{-1}ab=a^{-1}\rangle$ is a dihedral group of order $2m$, where $m>1$ is an odd square-free integer;
  \item[(b2)] $G=\widehat{D}_{4m}=\langle a,b\,|\, a^m=1, b^4=1, b^{-1}ab=a^{-1}\rangle$ of order $4m$ is a central extension of a dihedral group, where $m>1$ is an odd square-free integer;
  \item[(c)] $G$ is one of the groups $C_4$, $Q_8$, $Alt(4)$ (the alternating group of degree 4 and order 12).
\end{enumerate}
\end{theorem}

Since the groups in cases (b1), (b2), and (c) all have even order, we obtain a neat corollary.

\begin{corollary}\label{odd}
Every DCI-group of odd order is abelian with elementary abelian Sylow subgroups.
\end{corollary}

If the order of the coprime indecomposable group is divisible by exactly two distinct prime numbers, then we have a complete description of DCI-groups.

\begin{proposition}\label{pqDCI}
Let $G$ be a finite group such that its order has exactly two distinct prime divisors, and suppose that $G$ is coprime indecomposable (i.e., not nilpotent).
Then $G$ is a DCI-group iff $G$ is one of the groups $D_{2p}$, $\widehat{D}_{4p}$ (with an odd prime $p$), $Alt(4)$.
\end{proposition}

In Section 2 we will prove these results, here we just comment on the various cases appearing in Theorem~\ref{main}.

Let $G=C_p^r$ be an elementary abelian $p$-group of rank $r$. In a major work Yang-Quan Feng and Istv\'an Kov\'acs \cite{FengKovacs} proved that if $r\le 5$ then $C_p^r$ is a DCI-group. For $p=2$ this is best possible, as $C_2^6$ is not a CI-group, see Lewis Nowitz \cite{Nowitz}. For $p=3$ Pablo Spiga \cite{Spiga} gave an example showing that $C_3^8$ is not a CI-group, while the cases $r=6$ and $r=7$ are still open. For arbitrary primes G\'abor Somlai \cite{Somlai1} gave a very smart construction yielding that $C_p^{2p+3}$ is not a CI-group. He informed me that with a slight modification of his construction ChatGPT lowered the rank by 1, that is, $C_p^{2p+2}$ is not a CI-group either. The big gap between the positive and negative results warrants the following questions.

\begin{problem}
For a given prime $p$, what is the largest $r$ such that $C_p^r$ is a DCI-group? Is it bounded for all primes $p$ or the rank increases with $p$?
\end{problem}

The groups $D_{2m}$ and $\widehat{D}_{4m}$ in cases (b1) and (b2) of Theorem~\ref{main} have the same order as cyclic groups in Muzychuk's celebrated result (Theorem~\ref{Muzy}),
and contain a cyclic subgroup of index two.
In a recent preprint Kov\'acs and Somlai \cite{KovacsSomlai} settled the case of groups belonging to case (b1) by proving the following remarkable result:

\begin{theorem}(Kov\'acs and Somlai \cite{KovacsSomlai})
All dihedral groups $D_{2m}$ of square-free orders are DCI-groups.
\end{theorem}

It seems plausible that their central extensions $\widehat{D}_{4m}$ are DCI-groups as well.

\begin{conjecture}
For every odd square-free number $m$, the central extension $\widehat{D}_{4m}$ of the dihedral group is a DCI-group.
\end{conjecture}

Finally, note that the groups $C_4$, $Q_8$ and $Alt(4)$ are indeed DCI-groups \cite{ConderLi}.

\medskip

The proof of Proposition~\ref{pqDCI} and Theorem~\ref{main} that we will give in Section 2 is based on the fact that CI-groups are solvable.
This fundamental result on CI-groups has been reached in several steps.
After establishing the most basic properties of CI-groups in \cite{BabaiFrankl1} L\'aszl\'o Babai and P\'eter Frankl \cite{BabaiFrankl2} studied the structure of possible non-solvable CI-groups. Making use of some deep theorems from group theory (although not the classification of finite simple groups) they obtained severe restrictions on such groups. In particular, they showed that the only nonabelian composition factor of a CI-group can be either $Alt(5)$ or $\mathrm{PSL}(2,13)$. It was not difficult to exclude
the case $\mathrm{PSL}(2,13)$ (see \cite{LiPraeger}). Finally, Cai Heng Li \cite{Li1} constructed Cayley graphs of valency 29 showing that $Alt(5)$ is not a CI-group, thereby proving that all CI-groups are solvable. Later Marston Conder and Li \cite{ConderLi} found Cayley graphs of valency 5 for the same purpose, using computer search. Although the
the paper \cite{BabaiFrankl2} deals with CI-groups, and not with DCI-groups, the last sentence nevertheless says that ``From these results one can easily derive that every DCI-group is solvable". I don't see any easy way to obtain this result, so my proof will rely on the previously mentioned constructions showing that $Alt(5)$ is not a DCI-group.

In Section 3 I am going to sketch an alternative proof of Theorem~\ref{main}. This proof will be mainly group theoretic, but in contrast to the previous proofs it only makes use of classical theorems found in any textbook on group theory. The keystone of our reasoning will feature a nice counting argument borrowed from L\'aszl\'o R\'edei \cite{Redei}.

This paper is an extended version of my contribution presented at Eurocomb'25 \cite{PPP}.

\section{Proof of Theorem~\ref{main}}

We start by summarizing some well-known arguments.
It is easy to see that subgroups of DCI-groups are DCI-groups themselves.
Since the order of an element in a DCI-group cannot be divisible by $8$ or by a square of an odd prime (see Theorem~\ref{Muzy}), the exponent of a Sylow $p$-subgroup in a DCI-group is
either $p$ or $4$.
An obvious necessary condition for the DCI property of $G$ is that any two subgroups of the same order must be images of each other via automorphisms of $G$.
Applying this criterion to the Sylow $p$-subgroup, one obtains that it is either elementary abelian, or $C_4$, or $Q_8$ (cf.\ \cite{BabaiFrankl1}).
Another useful result is the following strengthening of \cite[Lemma 4.6]{BabaiFrankl1}:

\begin{lemma}\label{nonnilpotent}
Let $G$ be the semidirect product of an elementary abelian $p$-group $P$ (for some prime $p$) and a cyclic group $\langle b \rangle$ acting nontrivially on $P$,
where either $b$ has prime order $q$ ($q\ne p$), or $b$ has order $4$ and $p>2$.
If $G$ is a DCI-group then either $|P|=p>2$ and $b$ acts on $P$ by inversion (hence in this case $b$ has order $2$ or $4$, and $G=D_{2p}$ or $G=\widehat{D}_{4p}$), or $|P|=2^2$ and
$b$ has order $3$ (hence $G=Alt(4)$).
\end{lemma}

\begin{proof}
First assume that the action of $b$ on $P$ is irreducible (i.e., $P$ is a minimal normal subgroup in $G$).
Besides the cases in our lemma, the cited result of Babai and Frankl \cite[Lemma 4.6]{BabaiFrankl1} on CI-groups allows the following possibilities that we have to rule out here:
(1) $|P|=p$, $q=3$; (2) $|P|=p$, $a$ has order $4$ and acts faithfully on $P$ (in this case $p\equiv 1\pmod 4$); (3) $|P|=3^2$, $a$ has order $4$.

In cases (1) and (2) let $S=\{b\}$ and $T=\{b^{-1}\}$. Then the Cayley digraphs $Cay(G,S)$ and $Cay(G,T)$
are isomorphic, since they are both equal to $p$ disjoint directed cycles of length $3$ or $4$, respectively.
However, since $b$ and $b^{-1}$ act differently on $P$, there is no automorphism of $G$ mapping $S$ to $T$, so $G$ is not a DCI-group in these cases, indeed.

In case (3), since the action of $b$ on $P$ is irreducible, it must be faithful, as well. Then $b$ permutes the $8$ nontrivial elements of $P$ in two cycles. Let us take $a_1,a_2\in P$ from two different cycles of the action of $b$, and let $S=\{a_1,ba_1b^{-1}\}$ and $T=\{a_1,a_2\}$. Then the Cayley digraphs $Cay(G,S)$ and $Cay(G,T)$ are isomorphic, since they are both equal to the disjoint union of four copies of the box product of two directed $3$-cycles. However, there is clearly no automorphism of $G$ mapping $S$ to $T$, so we have also ruled out case (3).

Now consider the case when the action of $b$ is not irreducible on $P$. By Maschke's Theorem \cite[10.16]{Isaacs} we can decompose $P=P_1\times\dots\times P_k$ ($k\ge2$), where the action of $b$ is irreducible on each factor $P_i$ ($i=1,\dots,k$), and we may assume that $b$ acts nontrivially on $P_1$. We already know that either $|P_1|=p>2$ and $b$ acts on $P$ by inversion or $|P_1|=2^2$ and $b$ has order $3$. In the latter case we can take nontrivial elements $a_1\ne a_2 \in P_1$, $a_3\in P_2$, then the Cayley graphs $Cay(P\langle b\rangle,\{a_1,a_2\})$ and $Cay(P\langle b\rangle,\{a_1,a_3\})$ are isomorphic (each of them is a disjoint union of $3|P|/4$ undirected cycles of length $4$), but
$\{a_1,a_3\}$ is not the image of $\{a_1,a_2\}$ under any automorphism of $P\langle b\rangle$, since $a_1$ and $a_2$ are conjugate elements in this group, while $a_1$ and $a_3$ are not. Finally, assume that $b$ (of order 2 or 4) inverts the elements of $P_1$. Then it must act the same way on each $P_i$, since subgroups of the same order are isomorphic
in a DCI-group. Hence $b$ inverts all elements of $P$, i.e., $P\langle b\rangle$ is either a generalized dihedral group or its central extension. By the breakthrough result of  Dobson, Muzychuk and Spiga \cite{DobsonMuzySpiga} these groups are not DCI-groups, so the proof is complete.
\end{proof}

Now we prove Proposition~\ref{pqDCI} that will also serve as a lemma for the proof of Theorem~\ref{main}.

\begin{proof}[Proof of Proposition~\ref{pqDCI}]
The groups listed in Proposition~\ref{pqDCI} are indeed DCI-groups: for $D_{2p}$ see \cite{Babai}; for $\widehat{D}_{4p}$ see \cite{LiLuPPP} (the proof given there for the CI property also works for the DCI property); for $Alt(4)$  see \cite{ConderLi}.
To show the converse, let $G$ be a directly indecomposable group such that its order has exactly two distinct prime divisors. Checking the list of possible CI-groups given in \cite{LiLuPPP}, it remains to show that the following groups are not
DCI-groups: (1) semidirect products $M\rtimes\langle z\rangle$, where $M$ is an elementary abelian $p$-group of order at least $p^2$, $z$ has order 2 or 4, and $z$ acts by inverting elements of $M$; (2) semidirect products $M\rtimes\langle z\rangle$, where $M$ is an elementary abelian $p$-group, $z$ has order 3, and $z$ acts nontrivially on $M$ by raising every element to the same power. By Lemma~\ref{nonnilpotent} these groups are not DCI-groups, indeed.
\end{proof}

Now we turn to the general case.

\begin{proof}[Proof of Theorem~\ref{main}]
Let $G$ be a coprime indecomposable DCI-group, and let the prime divisors of $|G|$ be $p_1,\dots,p_k$, where $k\ge2$.
By the crucial result of Cai Heng Li \cite{Li1} every DCI-group is solvable. For solvable groups Philip Hall's classical theorem yields a system of Sylow subgroups $P_1$,\dots,$P_k$ such that for each prime divisor $p_i$ ($i=1,\dots,k$) a Sylow $p_i$-subgroup is chosen and $P_iP_j=P_jP_i$ holds for each pair $i\ne j$ (see \cite[Problem 3C.3]{Isaacs}).
Let us draw a graph $\Gamma$ with vertices $p_1,\dots,p_k$, where $\{p_i,p_j\}$ is an edge if the subgroup $P_iP_j$ is not nilpotent.
This graph is connected because $G$ is coprime indecomposable. Moreover, Proposition~\ref{pqDCI} gives that the prime 2 is incident to every edge, hence $\Gamma$ is a star.
Note that in the three cases of Proposition~\ref{pqDCI} the Sylow 2-subgroups are different: it is $C_2$ in $D_{2p}$; $C_4$ in $\widehat{D}_{4p}$; and $C_2^2$ in $Alt(4)$.
So if the Sylow 2-subgroup of $G$ is $C_2=\langle z\rangle$, then for every odd prime divisor $p$ of $|G|$ the Sylow $p$-subgroup has order $p$ and is inverted by conjugation with $z$. Since there is no edge between odd primes, the Sylow subgroups corresponding to the odd primes commute with each other, and their direct product is a cyclic normal subgroup of odd square-free order, inverted by $z$. Thus we obtain $G=D_{2m}$ as described in case (2a) of Theorem~\ref{main}. Similarly, if the Sylow 2-subgroup of $G$ is $C_4$, then we
obtain $G=\widehat{D}_{4m}$ as in case (2b). Finally, if the Sylow 2-subgroup of $G$ is $C_2^2$, then $\Gamma$ has just one edge, and $G=Alt(4)$.
\end{proof}

\section{An alternative proof}

As it was explained in Section~1, the solvability of CI-groups has been established applying some deep results from modern group theory. So it seems worthwhile to give a proof
for Theorem~\ref{main} that does not use any heavy machinery from group theory. At the very end of our proof we will have to eliminate $Alt(5)$ from the list of possible DCI-groups, where we can refer to previous constructions due to Li \cite{Li1} and Conder and Li \cite{ConderLi}.

First we prove Corollary~\ref{odd}. Let $G$ be a DCI-group of odd order. Then every Sylow subgroup of $G$ is elementary abelian. 
Suppose that for some Sylow $p$-subgroup $P$ of $G$, the normalizer of $P$ is strictly larger that its centralizer. Then $N_G(P)$ contains an element $b$ of prime order $q$
($q\ne p$) acting nontrivially on $P$. However, it is not possible by Lemma~\ref{nonnilpotent}. Hence the normalizer and the centralizer of $P$ coincide, so Burnside's Theorem on normal $p$-complements (see \cite[Theorem 5.13]{Isaacs}) yields
a normal subgroup $N_p\vartriangleleft G$ such that $G=PN_p$ and $P\cap N_p=1$. This holds for every prime divisor $p$ of $|G|$, hence it follows that $G$ is the direct
product of its Sylow subgroups, as we have claimed.

Now let $G$ be a coprime indecomposable DCI-group of even order, and let $S$ be a Sylow 2-subgroup of $G$. We know that $S$ is either elementary abelian or cyclic of order $4$ or the quaternion group $Q_8$. Suppose that $G$ contains a normal 2-complement $N$. Then $N$ is a DCI-group of odd order, hence abelian. Let $P$ be a Sylow $p$-subgroup of $N$ ($p$ an odd prime), then $P$ is normal in $G$.
Let $b\in S$ be an arbitrary element. Looking at $P\langle b \rangle$, Lemma~\ref{nonnilpotent} yields that either $b$ acts trivially on $P$ or $|P|=p$ and $b$ acts by inverting elements of $P$. Taking into account that subgroups of the same order must be isomorphic in a DCI-group, it follows that elements of the same order in $S$ must act in
the same way on $P$. Hence either the whole $S$ acts trivially on $P$ or $|P|=p$ and $S$ is cyclic (of order 2 or 4) and its generator inverts the elements of $P$. In the first
case $P$ is contained in the center of $G$, so $G$ is the direct product of $P$ and the normal $p$-complement, contrary to the coprime indecomposability of $G$. Hence we get that every Sylow subgroup of $N$ is cyclic of prime order, so $N$ itself is cyclic of square-free order, and the generator of the cyclic Sylow 2-subgroup $S$ acts on $N$ by inverting elements, so $G$ is indeed a dihedral or extended dihedral group as described in cases (b1) and (b2) of Theorem~\ref{main}.

Next we show that $G$ necessarily contains a normal 2-complement, unless $S$ is elementary abelian of order 4. If $S$ is cyclic, then it is an elementary exercise. (It also follows
from Burnside's Theorem.) If $S=Q_8$, then we can invoke a theorem of Frobenius (see \cite[Theorem 5.26]{Isaacs}), that a normal $p$-complement exists iff
$N_G(X)/C_G(X)$ is a $p$-group for every $p$-subgroup $X\le G$. (Although it is not easy to prove this theorem in full generality, for the case $S=Q_8$ it can be derived quite easily.) Now for cyclic subgroups $X$ the condition is automatically fulfilled, so we obtain that $G$ has a normal 2-complement iff $N_G(S)/C_G(S)$ is a 2-group. Here
$N_G(S)/C_G(S)$ is isomorphic to a subgroup of $Aut(Q_8)\cong Sym(4)$, so it is a 2-group, unless there is an element $b$ of order 3 cyclically permuting the three cyclic subgroups of order 4 in $Q_8$. But in the latter case the subgroup $S\langle b \rangle\cong \mathrm{SL}(2,3)$ is not a DCI-group as shown in \cite{ConderLi}. So a normal 2-complement does exist in this case as well. Finally, if $S$ is an elementary abelian 2-group of order at least 8, then Lemma~\ref{nonnilpotent} can be applied again to get
$N_G(S)=C_G(S)$, so Burnside's Theorem yields a normal 2-complement also in this case.

So it remains to investigate the case when the Sylow 2-subgroup $S$ of $G$ is elementary abelian of order 4, and $G$ does not contain a normal 2-complement. Then we must have
$N_G(S)>C_G(S)$, so there is an element of order 3 in the normalizer of $S$ acting nontrivially.

Let $p\ge5$ be a prime divisor of the order of $G$, and $P$ a Sylow $p$-subgroup. We know that $P$ is elementary abelian. If $G$ contains a normal $p$-complement $N$, then the Frattini argument \cite[Lemma 1.13]{Isaacs} yields that for every prime $r \mid |N|$ there exists a Sylow $r$-subgroup $R$ normalized by $P$. However, by Lemma~\ref{nonnilpotent} $P$ must centralize $R$, hence $G=N\times P$, contrary to the assumption that $G$ is coprime indecomposable. So we obtain that $G$ does not contain a normal $p$-complement, hence by Burnside's theorem we must have $N_G(P)>C_G(P)$. Now Lemma~\ref{nonnilpotent} gives that $|P|=p$ and $|N_G(P)/C_G(P)|=2$. Note also, that because $N_G(P)$ contains a dihedral group of order $2p$, there cannot be any element of order $2p$ in $G$.

For $p=3$ it is a little more involved to reach the same conclusion. Suppose that $G$ contains a normal $3$-complement $N$. Let $S$ be a Sylow 2-subgroup in $N$ normalized by the
Sylow 3-subgroup $P$. As $SP$ is a DCI-group, we obtain $|P|=3$ and $SP\cong Alt(4)$. Since 3 does not divide the order of $N$, the above arguments yield that $N$ contains a normal $2$-complement $M$, and $M$ is abelian. For every element $x\in SP$ of order 3, the subgroup $M\langle x\rangle$ is a DCI-group of odd order, hence $x\in C_G(M)$, therefore $M$ lies in the center of $G$, $G=M\times SP$. Thus $G=Alt(4)$ in this case (belonging to part (c) of Theorem~\ref{main}), since $G$ is coprime indecomposable. If this is not the case, then we obtain the same way as for the prime divisors greater than 3, that the Sylow 3-subgroup has order 3, its centralizer has index 2 in the normalizer, and there is no element of order 6 in $G$.

Let $\Pi$ be the Gruenberg--Kegel graph of $G$, with vertices representing the prime divisors of $|G|$ and edges drawn between primes if there is an element of the group whose order is the product of the two primes. Our previous observations show that 2 is an isolated vertex in $\Pi$. We claim that every connected component of $\Pi$ is a complete graph. Indeed, if $\{p_1,p_2\}$ and $\{p_2,p_3\}$ are edges of the graph, then the centralizer of a Sylow $p_2$-subgroup contains elements of order $p_1$ and $p_3$. This centralizer is a DCI-group of odd order, hence it is abelian by Corollary~\ref{odd}, so the elements of order $p_1$ and $p_3$ commute, guaranteeing the edge $\{p_1,p_3\}$. Now
take the decomposition of the vertex set of $\Pi$ into connected components $\Pi=\{2\}\cup\Pi_1\cup\dots\cup\Pi_k$, and let $n_i$ denote the product of the primes in $\Pi_i$.
Let $n=|G|=4\prod_{i=1}^{k} n_i$. Here $n_1,\dots,n_k$ are pairwise coprime odd square-free numbers.
Our previous observations yield that if $p\in \Pi_i$, then the normalizer of a Sylow $p$-subgroup is a dihedral group of order $2n_i$. Furthermore, the normalizer of a Sylow 2-subgroup is $Alt(4)$.

Now we are going to show that the collection of all Sylow 2-subgroups together with all cyclic subgroups of order $n_i$ ($i=1,\dots,k$) forms a partition of $G$, meaning that every nontrivial element of $G$ belongs to exactly one of these subgroups.
Let $1\ne x\in G$. Then a suitable power $y$ of $x$ has prime order and $x\in C_G(y)$. If $y$ has order 2, then $C_G(y)$ does not contain any nontrivial element of odd order,
hence $C_G(y)$ is the unique Sylow 2-subgroup containing $x=y$.
If $y$ has odd prime order $p\in \Pi_i$, then $C_G(y)$ is a DCI-group of odd order, hence it is abelian, and since with the exception of Sylow 2-subgroups, all Sylow subgroups of $G$ have prime order, $C_G(y)$ is cyclic of order $n_i$, and it is the unique subgroup in our system that contains $x$. Clearly, all cyclic subgroups of order $n_i$ are conjugate
to each other, hence their number is $n/2n_i$. The number of Sylow 2-subgroups is $n/12$. As we have partitioned $G$, we obtain
\[
n-1=\frac{n}{12}\cdot 3 + \sum_{i=1}^{k} \frac{n}{2n_i}(n_i-1),
\]
and dividing by $n$ we get
\[
1-\frac{1}{n} = \frac{1}{4} + \frac{k}{2} - \sum_{i=1}^{k} \frac{1}{2n_i}.
\]
If $k=1$, then the left hand side is greater than 3/4, while the right hand side is smaller. If $k\ge 3$, then the left hand side is less than 1, while the
right hand side is greater than $\frac{1}{4}+3\left(\frac{1}{2}-\frac{1}{6}\right)>1$. Hence only $k=2$ is possible. Then we obtain $4n_1n_2-1=n_1n_2+2n_2(n_1-1)+2n_1(n_2-1)$,
that is $(n_1-2)(n_2-2)=3$, so $n_1=3$, $n_2=5$ (or vice versa), $|G|=60$ and $G$ contains five subgroups isomorphic to $Alt(4)$, hence $G\cong Alt(5)$. However,
$Alt(5)$ is not a DCI-group (see \cite{Li1} or \cite{ConderLi}), finishing the proof.

\end{document}